\documentclass[12pt]{amsart}
\usepackage{amssymb}
\usepackage{amsmath}
\usepackage{bbm}
\usepackage{amsfonts,mathrsfs}
\usepackage{mathtools}
\usepackage[dvips]{graphicx}
\usepackage{hyperref}
\usepackage[capitalize,nameinlink,noabbrev,nosort]{cleveref}
\usepackage[dvipsnames]{xcolor}
\usepackage{todonotes}
\usepackage[shortlabels]{enumitem}
\usepackage{ytableau}

\definecolor{dblue}  {RGB}{20,66,129}
\definecolor{jpurple}{HTML}{684B81}
\definecolor{jgreen} {HTML}{00846E}

\usetikzlibrary{automata,arrows,positioning,graphs,graphs.standard,decorations.markings,decorations.pathreplacing,shapes.geometric}
\tikzset{
    eo/.style={
        line width=2pt,
        dblue!60
    },
    fo/.style={ 
        decoration={
            markings,
            mark=at position #1 with {\arrow[line width=3pt]{>}}
        },
        postaction={decorate},
        eo
    },
    fo/.default=0.58,
    ro/.style={ 
        decoration={
            markings,
            mark=at position #1 with {\arrow[line width=3pt]{<}}
        },
        postaction={decorate},
        eo
    },
    ro/.default=0.58,
    main/.style={
        circle, 
        fill, 
        inner sep=0pt, 
        minimum size=10pt, 
        text=white,
        draw=dblue!70,
        fill=dblue!70
    },
    arrow/.style={
        ->,
        line width=2pt
    },
    drill/.style={draw=#1,fill=#1},
    col2/.style={drill=jgreen!70},
    col3/.style={drill=jpurple}
}

\numberwithin{equation}{section}

\newcommand{\ds}{\displaystyle}
\DeclareMathOperator{\id}{\text{id}}

\DeclarePairedDelimiter\abs{\lvert}{\rvert}
\DeclarePairedDelimiter\norm{\lVert}{\rVert}%

\makeatletter
\let\oldabs\abs
\def\abs{\@ifstar{\oldabs}{\oldabs*}}

\let\oldnorm\norm
\def\norm{\@ifstar{\oldnorm}{\oldnorm*}}
\makeatother

\newcommand{\conf}[2]{#1^{(#2)}}
\newcommand{\blank}{\rule{0.2cm}{0.15mm}}
\newcommand{\floor}[1]{\lfloor #1 \rfloor}
\newcommand{\ceil}[1]{\lceil #1 \rceil}
\newcommand*\ex{\text{ex}}
\DeclareRobustCommand{\stirling}{\genfrac\{\}{0pt}{}}

\newtheorem{thm}{Theorem}[section]
\newtheorem{prop}[thm]{Proposition}

\newtheorem{lem}[thm]{Lemma}
\newtheorem{defn}[thm]{Definition}
\newtheorem{rem}[thm]{Remark}

\newtheorem{eg}[thm]{Example}

\title[Toppleable Permutations, Excedances, Acyclic Orientations]
{Toppleable Permutations, Excedances and Acyclic Orientations}

\author{Arvind Ayyer}
\address{Arvind Ayyer, Department of Mathematics, Indian Institute of Science, Bangalore  560012, India.}
\email{arvind@iisc.ac.in}

\author{Daniel Hathcock}
\address{Daniel Hathcock, Department of Mathematical Sciences, 
Carnegie Mellon University, Pittsburgh, U.S.A.}
\email{dhathcoc@andrew.cmu.edu}

\author{Prasad Tetali}
\address{Prasad Tetali, School of Mathematics and School of Computer Science, Georgia Institute of Technology, Atlanta, GA 30332, U.S.A. }
\email{tetali@math.gatech.edu}


\date{\today}

\begin{document}

\begin{abstract}
Recall that an excedance of a permutation $\pi$ is any position $i$
such that $\pi_i > i$. Inspired by the work of Hopkins, McConville and
Propp (Elec. J. Comb., 2017) on sorting using toppling, we say that
a permutation is toppleable if it gets sorted by a certain sequence of
toppling moves. One of our main results is that the number of toppleable permutations on $n$ letters is the same as those for which excedances happen exactly at $\{1,\dots, \lfloor (n-1)/2 \rfloor\}$. Additionally, we show that the above is also the number of acyclic
orientations with unique sink (AUSOs) of the
complete bipartite graph $K_{\lceil n/2 \rceil, \lfloor n/2 \rfloor + 1}$.
We also give a formula for the number of AUSOs of complete multipartite graphs. We conclude with observations on an extremal question of Cameron et al. concerning maximizers of (the number of) acyclic orientations, given a prescribed number of vertices and edges for the graph. 
\end{abstract}

\maketitle

\section{Introduction}

The sandpile model on graphs has been extensively studied, ever since its introduction by Dhar~\cite{dhar-1990} and independently by Bj\"orner, Lovasz and Shor~\cite{bjorner-lovasz-shor-1991} (where it is called the chip-firing game). From the mathematical standpoint, this model has influenced subjects as diverse as combinatorics, probability theory, algebraic geometry and combinatorial commutative algebra. See the recent book by Klivans~\cite{klivans-2019} for details.

In the usual sandpile model, all chips are indistinguishable. An impetus to studying the sandpile model with labeled chips is the recent work of Hopkins, McConville and Propp~\cite{hopkins2016sorting}, who studied the model on $\mathbb{Z}$ with chips labeled $1,\dots,n$ initially at the origin. We use the term {\em toppling} to describe this dynamics.
For $n$ even, they showed that the resulting configuration is always sorted. But for $n$ odd, they conjectured based on simulations that the probability of being sorted tends to $1/3$ as $n \to \infty$. This conjecture has motivated a lot of work since then~\cite{galashin-et-al-2018, galashin-et-al-2019, hopkins-postnikov-2019, klivans-liscio-2020, klivans-felzenszwalb-2021, klivans-liscio-2020+}.

In this work, we consider the model on $\mathbb{Z}$ with initial configurations motivated by the analysis of certain special sequences of moves in the original model. We show that the enumerative combinatorics of 
those initial configurations which get sorted is very rich and related to classical permutation enumerations via the excedance statistic. We also show bijectively that these objects are related to the number of acyclic orientations with a unique sink (AUSOs) of complete bipartite graphs. 
The relation between acyclic orientations of bipartite graphs and the excedance statistic has also been studied in the context of so-called poly-Bernoulli numbers~\cite{benyi-hajnal-2017}.
We then extend our explicit enumeration result on AUSOs of complete bipartite graphs to the $N$-partite case (for $N\ge 2$), with prescribed sizes for  the $N$ parts.

Further motivated by a result in \cite{hopkins2016sorting}, we study a new class of permutations which we call {\em collapsed permutations}, and show that they are enumerated by the classical Genocchi numbers of the first and second kind. 
Our final contribution in this work briefly considers a conjecture of Cameron, Glass and Schumacher~\cite{cameron2014acyclic} that Tur\'an graphs maximize the number of acyclic orientations over graphs with given numbers of vertices and edges. We show that when $m \geq \binom{n}{2} - \floor{\frac{n}{2}}$, a Tur\'an graph whose parts have size 1 or 2 maximizes the number of AOs over all graphs on $n$ vertices and $m$ edges.

The plan of the rest of the paper is as follows. In \cref{sec:topp}, we motivate and introduce the toppling model on permutations, prove some of its basic properties, and state the main results. In \cref{sec:topp=exc}, we prove one of the main results stated in the previous section. In \cref{sec:collap}, we define the set of collapsed permutations and describe the connections to the Genocchi numbers.
In \cref{sec:complete-bipartite}, we prove the second main result of \cref{sec:topp}.
\cref{sec:ao-enum} is devoted to proving enumeration formulas for AUSOs of complete multipartite graphs. 
These can also be obtained using the generating function for the chromatic polynomial of complete multipartite graphs given in \cite[Solution to Problem~5.6]{stanley-ec2}; see \cref{rem:stanley} for more details.
\cref{sec:cameron} discusses the extremal question of maximizing the number of acyclic orientations with given numbers of vertices and edges.

\section{Toppling on permutations}
\label{sec:topp}

For $n \geq 2$, we will consider labeled chip configurations on the segment $L_n = \{-\lfloor (n+1)/2 \rfloor, \dots, -1, 0, 1, \dots, \lfloor n/2 \rfloor + 1\} $. 
We will think of $0$ as the origin.
At each position in $L_n$, we will place a certain number of labeled chips, with a total number of $n$ chips labeled $1$ through $n$.
We will consider the discrete dynamical system called {\em toppling}, defined in \cite{hopkins2016sorting} (called chip-firing therein), on configurations in $L_n$ as follows:
\begin{enumerate}
    \item If no position in $L_n$ has two or more chips, stop. Else, go to step 2.
    \item Choose a position $i$ uniformly at random among positions occupied by more than one chip.
    \item Pick two chips $\alpha < \beta$ uniformly from those at site $i$. 
    \item Move $\alpha$ to position $i-1$ and $\beta$ to $i+1$. 
    \item Go to step 1.
\end{enumerate}

We will be interested in a special class of initial configurations arising from permutations as follows. As usual, let $[n]$ denote $\{1,\dots,n\}$.
For $\pi = (\pi_1,\dots,\pi_n) \in S_n$ and an element $r \in [n+1]$, 
we will place the elements $\pi_1,\dots,\pi_n$ in positions $-\lfloor (n-1)/2 \rfloor, \dots, -1, 0, 1, \dots, \lfloor n/2 \rfloor$, 
shift the labels of elements in $\pi$ greater than or equal to $r$ by $1$, and place $r$ at the origin. We will call this initial condition $\conf{\pi}{r}$.
For example, with $r = 2$ and permutations $\rho = (3,1,4,2) \in S_4$ and $\sigma = (2,5,1,3,4) \in S_5$, we obtain 
\begin{equation}
\label{eg-conf}
\conf{\rho}{2} = \;
\begin{matrix}
&&1 \\
& 4 & 2 & 5 & 3 \\
\hline
-2 & -1 & 0 & 1 & 2 & 3
\end{matrix} \;\;\;,
\quad
\conf{\sigma}2 = \;
\begin{matrix}
&&& 1 \\
& 3 & 6 & 2 & 4 & 5 \\
\hline
-3 & -2 & -1 & 0 & 1 & 2 & 3
\end{matrix}\; \; \;.
\end{equation}
For convenience, we will denote configurations in one-line notation similar to that used for permutations, except that when there are multiple chips at a site, we will enclose them in parenthesis. For the above examples, we will write $\conf{\rho}2 = (4,(1,2),5,3)$ and $\conf{\sigma}2 = (3, 6, (1,2), 4, 5)$. The vacancies in the first and last position will be understood.

The motivation for studying this model comes from the toppling dynamical system on $\mathbb{Z}$ with $n \geq 2$ chips initially at the origin, studied in detail in \cite{hopkins2016sorting}. Fix an $r \in [n]$ and consider a special case of topplings of this model where we do not permit $r$ to be picked as one of two chips to be moved. If $n$ is odd, $r$ will end up by itself at the origin and the process will terminate. If $n$ is even, something more interesting happens. At the end of these topplings, $r$ will have a partner at the origin and we will end up with a configuration $\conf{\pi}{r}$ in $L_{n-1}$ for some $\pi \in S_{n-1}$.
If we now allow $r$ to topple, all the chips will be {\em sorted}, thanks to \cite[Theorem 13]{hopkins2016sorting}. This tells us that for a large class of permutations $\pi$, the final configuration after toppling $\conf{\pi}{r}$ is sorted. 
Thus, it is natural to try to understand the set of all permutations for which sorting happens.

We begin with some basic properties of the toppling dynamics in $L_n$ starting with $\conf{\pi}r$.

\begin{prop}
\label{prop:determ}
Fix $\pi \in S_n$ and $r \in [n+1]$.
The toppling dynamical system on $L_n$ with initial condition $\conf{\pi}r$ satisfies the following properties.
\begin{enumerate}
    \item The final configuration is deterministic.
    \item At every step, the configuration lives in $L_n$. In other words, no chip leaves $L_n$.
    \item In the final configuration, there is precisely one chip at every position in $L_n$, except the origin (resp. position 1) when $n$ is odd (resp. even). 
\end{enumerate}
\end{prop}

\begin{proof}
We begin with (1) by showing that at each point in time, for any two positions $a < b \in L_n$ each containing two chips, there must be some $c$ with $a < c < b$ and no chips at $c$. In particular, this implies that no position can ever have more than 2 chips.

The initial configuration $\conf{\pi}{r}$ satisfies this condition vacuously, as there is only a single position containing 2 chips. Then inducting on the time step $t$, suppose we toppled position $a$ at time $t$. Looking at position $a + 1$, it could have contained either 0 or 1 chip before toppling $a$, by the induction hypothesis. If $a + 1$ contained 1 chip, then it now contains 2 chips. But then for any $b > a$ containing 2 chips, the induction hypothesis gives that there is an empty position between $a$ and $b$, and it cannot have been $a + 1$, so that empty space separates $a +1$ from any such $b$. Likewise, any position to the left of $a + 1$ containing 2 chips is separated from $a + 1$ by the now-empty position $a$. Otherwise, $a + 1$ did not contain a chip. Then if $a + 1$ was the empty space separating any two positions with 2 chips, the now-empty position $a$ now separates them. The argument for position $a-1$ is identical. 

From (1), no position contains more than 2 chips. The only way a chip can leave $L_n$ is if the leftmost or rightmost positions contain more than one chip. Focus on the position $\floor{n/2}$. When it contains two chips, it transfers one chip to the rightmost position $\floor{n/2}+1$ and itself becomes empty. Arguing inductively, one can see that from that point on, $\floor{n/2}$ can never contain more than one chip. 
The argument for the leftmost position is similar. This proves (2).

The argument for (3) is similar to that of (2). The base cases of $n=2,3$ are easily checked. Once there is a chip at positions $-\floor{(n+1)/2}$ and $\floor{n/2+1}$, the adjoining sites are empty and we end up with a configuration in $L_{n-2}$.
\end{proof}

By \cref{prop:determ}(3), the final configuration can be interpreted as a permutation in $S_{n+1}$ and by \cref{prop:determ}(1), the toppling dynamical system on $L_n$ can be considered as a map 
$\mathcal{T}: S_n \times [n+1] \to S_{n+1}$.
Let $\id$ be the identity (namely sorted) permutation whose size will be clear from the context. 

\begin{defn}
\label{def:topp}
We say that a permutation $\pi$ is {\em $r$-toppleable} if $\mathcal{T}(\pi,r) \allowbreak = \id$, and we say that $\pi$ is {\em toppleable} if $\pi$ is $r$-toppleable for all $r \in [n+1]$. 
\end{defn}

For the examples in \eqref{eg-conf}, the final configurations can be seen to be $\mathcal{T}(\rho,2) = (1,2,3,4,5)$ and $\mathcal{T}(\sigma,2) = (1,2,3,4,6,5)$. Therefore, $\rho$ is $2$-toppleable, but $\sigma$ is not.
It is easy to see that in general, $\mathcal{T}(\id,r) = \id$ for all $r$ and therefore $\id$ is always toppleable.
The following symmetry property of toppling dynamics can be easily seen by studying what happens during a single toppling move.

\begin{prop}
\label{prop:sym}
Suppose $n \geq 3$ is odd, $r \in [n+1]$ and $\pi = (\pi_1,\dots,\allowbreak \pi_n) \in S_n$. Let 
$\hat\pi = (n+1-\pi_n,\dots,n+1-\pi_1)$. Then the toppling dynamics on $\conf{\pi}{r}$
is isomorphic to that on $\conf{\hat\pi}{n+2-r}$ via the map which reflects configurations about the origin and interchanges chip $i$ with $n+2-i$.

Since $\widehat\id = \id$, $\pi$ is $r$-toppleable if and only if $\hat\pi$ is $(n+2-r)$-toppleable.
\end{prop}

Let $t_r(n)$ be the number of $r$-toppleable permutations and $t(n)$ be the number of toppleable permutations in $S_n$.
For example, $t_1(3) = t_4(3) = 4$ since there are four $1$-toppleable permutations, namely $123$, $213$, $132$ and $231$, as well as four $4$-toppleable permutations, namely $123$, $213$, $132$ and $312$, in $S_3$.
The common permutations among these turn out also to be $2$- and $3$-toppleable and hence $t(3) = 3$.
Data for $t_r(n)$ for small values of $r$ and $n$ is given in \cref{tab:data-rtopp}.
As expected from \cref{prop:sym}, $t_i(n) = t_{n+2-i}(n)$ for $n$ odd.

\begin{table}[htbp!]
\begin{tabular}{c|ccccccccc}
\hline
$n$ \textbackslash $r$ & 1 & 2 & 3 & 4 & 5 & 6 & 7 & 8 & 9 \\
\hline
3 & 4 & 3 & 3 & 4 \\
4 & 14 & 10 & 7 & 7 & 8 \\
5 & 46 & 38 & 31 & 31 & 38 & 46 \\
6 & 230 & 184 & 146 & 115 & 115 & 130 & 146 \\
7 & 1066 & 920 & 790 & 675 & 675 & 790 & 920 & 1066 \\
8 & 6902 & 5836 & 4916 & 4126 & 3451 & 3451 & 3842 & 4264 & 4718 \\
\hline
\end{tabular}
\vspace{0.5cm}
\caption{The number of $r$-toppleable permutations, $t_r(n)$, for $3 \leq n \leq 8$.}
\label{tab:data-rtopp}
\end{table}

Recall that an {\em excedance} of a permutation $\pi$ is any position $i$ such that $\pi_i > i$. The positions at which there are excedances for $\pi$ is called the {\em excedance set} of $\pi$.
For example, the permutations $\rho$ and $\sigma$ considered in \eqref{eg-conf} have excedance sets
$\{1,3\}$ and $\{1,2\}$ respectively. 
There have been a lot of studies of the excedance statistic on permutations.
What will be relevant to us is the study of permutations whose excedance set is $\{1,\dots,k\}$ for some $k$. This was initiated by Ehrenborg and Steingr\'imsson~\cite{ehrenborg-steingrimsson-2000}, who gave a formula for the number $a_{n,k}$ of such permutations in $S_n$. 
The bivariate exponential generating function of $a_{r+s,s}$ is given by the explicit formula~\cite[Theorem 3.1]{ehrenborg-clark-2010}
\begin{equation}
\label{egf-exc}
\sum_{r,s \geq 0} a_{r+s,s} \frac{x^r}{r!} \frac{y^s}{s!} 
= \frac{e^{-x-y}}{(e^{-x} + e^{-y} - 1)^2}.
\end{equation}

\begin{thm}
\label{thm:topp=exc}
For all $n$, the number of toppleable permutations in $S_n$ satisfies $t(n) = t_{\floor{n/2}+1}(n) = t_{\floor{n/2}+ 2}(n)$. Furthermore, this number is given by $t(n) = a_{n,\floor{(n-1)/2}}$.
\end{thm}

Using \eqref{egf-exc}, de Andrade, Lundberg and Nagle~\cite[Theorem 1.2]{de-andrade-etal-2015} obtained the asymptotic formula,
\[
t(n) = \frac{1}{2 \log 2 \sqrt{1 - \log 2} + o(1)}   \frac{n!}{(2 \log 2)^n}  .
\]

These numbers are given by the central diagonal chips in the triangle \cite[Sequence A136126]{OEIS} of the OEIS. Notice that the latter triangle is symmetric and therefore, the answer is unambiguous for even $n$.

To state our next result, we recall the notion of acyclic orientations.
For any simple, undirected graph, an {\em orientation} is an assignment of arrows to the edges. An {\em acyclic orientation} (AO) is an orientation in which there is no directed cycle. It is easy to see that every graph has an acyclic orientation and every acyclic orientation has at least one {\em source} (vertex with no incoming arrows) and one {\em sink}
(vertex with no outgoing arrows).
An {\em acyclic orientation with a unique sink} (AUSO), also known as the {\em Ursell function} is an acyclic orientation with exactly one sink.
Stanley showed that the number of acyclic orientations of any graph (up to sign) is given by the chromatic polynomial of the graph evaluated at $-1$~\cite{stanley}. A related result of Greene and Zaslavsky~\cite{greenezaslavsky} is that the number of acyclic orientations with a unique sink is independent of the sink and equal to (again up to sign) the linear coefficient of the chromatic polynomial.

Our focus here will be on AUSOs of complete bipartite graphs $K_{m,n}$.

\begin{thm}
\label{thm:topp=ursell}
The number of toppleable permutations in $S_n$, $t(n)$, is the same as the number of acyclic orientations with a fixed unique sink of $K_{\ceil{n/2},\floor{n/2}+1}$.
\end{thm}

\section{Toppleable permutations and excedances}
\label{sec:topp=exc}

In this section, we will prove \cref{thm:topp=exc}. 
To make the presentation cleaner, we state the results separately for odd and even $n$ for the most part. This will avoid the presence of floors and ceilings all over the place.
We begin with a monotonicity result. 

\begin{thm}
\label{thm:certi}
Let $n = 2m+1$ and $\pi \in S_n$. 
\begin{enumerate}
\item Suppose $2 \leq r \leq m+1$. Then $\pi$ is $(r-1)$-toppleable
if $\pi$ is $r$-toppleable. 

\item Suppose $m+2 \leq r \leq 2m$. Then $\pi$ is $(r+1)$-toppleable
if $\pi$ is $r$-toppleable. 

\item For $r = m + 1$, $\pi$ is $r$-toppleable if and only if $\pi$ is $(r + 1)$-toppleable. 
\end{enumerate}

Let $n = 2m$ and $\pi \in S_{n}$. 
\begin{enumerate}
    \item Suppose $2 \leq r \leq m + 1$. Then $\pi$ is $(r-1)$-toppleable if $\pi$ is $r$-toppleable. 
    
    \item Suppose $m + 2 \leq r \leq n$. Then $\pi$ is $(r+1)$-toppleable if $\pi$ is $r$-toppleable. 
    
    \item For $r = m + 1$, $\pi$ is $r$-toppleable if and only if $\pi$ is $(r + 1)$-toppleable. 
\end{enumerate}
\end{thm}

To see that the converse of the first statement is not true, consider the following example with $m=2$ and $r=3$. When $\pi = 24135$, 
$\conf{\pi}3 = (2, 5, (1,3), 4, 6)$ and $\conf{\pi}2 = (3, 5, (1,2), 4, 6)$ couple eventually. But 
when $\pi = 13452$, $\conf{\pi}3  = (1, 4, (3, 5), 6, 2)$ and $\conf{\pi}2 = (1, 4, (2, 5), 6, 3)$ never couple and we have
$\mathcal{T}(\conf{\pi}2) = 123456$, but $\mathcal{T}(\conf{\pi}3) = 132456$.

To prove \cref{thm:certi}, we need a lemma: 
\begin{lem}
\label{lem:final-move}
Let $\pi \in S_n$ and suppose $\pi$ is $r$-toppleable. Then 
\begin{enumerate}
    \item for each $1 \leq k \leq \floor{\frac{n}{2}}+1$, the final move of chip $k$ when toppling $\pi^{(r)}$ is to the left;
    
    \item for each $\floor{\frac{n}{2}}+2 \leq k \leq n + 1$, the final move of chip $k$ when toppling $\pi^{(r)}$ is to the right.
    
    \item in the final move, chips $\floor{\frac{n}{2}}+1$ and $\floor{\frac{n}{2}}+2$ topple to their correct positions. 
\end{enumerate}
\end{lem}

\begin{proof}
To prove (1), we use induction on $k$. First, since $\conf{\pi}{r}$ topples to the identity, chip 1 ends in the leftmost position of $L_n$. Then, Proposition \ref{prop:determ}(2) ensures that the final move of 1 into this position must be to the left, since otherwise the chip would have to lie outside $L_n$ directly before its final move. 

Now, using the induction hypothesis, we have that the final move of chip $k-1$ is to the left, say into position $p-1$. This leaves an empty space in position $p$. Moreover, after this final move, no chip may land on position $p-1$. We know, since $\pi$ is $r$-toppleable, that position $p$ must eventually hold chip $k$. And since chip $k$ cannot land on position $p-1$, it must make its final move to the left from position $p+1$. 

(2) is proved similarly. For (3), recall Proposition \ref{prop:determ}(3) gives that the final configuration contains no chips at position 0 (resp.\ 1) if $n$ is odd (resp.\ even). Since $\conf{\pi}{r}$ topples to the identity, we have $\floor{\frac{n}{2}}+1$ directly left of this empty position, and $\floor{\frac{n}{2}}+2$ directly to the right. It is clear that the only way to arrive at this configuration is for the final topple to have occurred at position 0 (resp.\ 1) containing these two chips.
\end{proof}

\begin{proof}[Proof of Theorem \ref{thm:certi}]

Our strategy for the proof is as follows. Suppose $\pi \in S_n$ is $r$-toppleable. 
Then, the only difference between $\conf{\pi}{r}$ and $\conf{\pi}{r-1}$ is that the positions of $r-1$ and $r$ are interchanged.
By definition, $r$ is positioned at the origin of $L_n$ in $\conf{\pi}{r}$ and let $j$ be the position of $r-1$ in  $\conf{\pi}{r}$.
If $j = 0$, then  $\conf{\pi}{r} =  \conf{\pi}{r-1}$ and the result trivially holds. If not, there are two possibilities.
Either $j > 0$ or $j < 0$. We might need different arguments in both cases.

At each step of the toppling procedure, $\conf{\pi}{r}$ and $\conf{\pi}{r-1}$ continue to differ only in their 
positions of $r-1$ and $r$. This will be the case until we reach the point when $r-1$ and $r$ are
at the same position. At this point, the toppling procedure is coupled and the final result is identity.
The only problem with this argument is that we could have reached the final result without ever being 
coupled. 

Now we assume that $2 \leq r \leq m +1$.
For the proof of statement (1), we will track the positions of $r-1$ and $r$ as time evolves only in $\conf{\pi}{r}$. If we can show that there is a time when $r-1$ and $r$ are at the same site, then we are done since $\conf{\pi}{r}$ and $\conf{\pi}{r-1}$ will be coupled.

If $j > 0$, that means $r-1$ is to the right of $r$ in $\conf{\pi}{r}$. But we know that eventually $r-1$ will end up to the left of $r$, since $\pi$ is $r$-toppleable. Therefore, there will necessarily be a time when $r-1$ and $r$ are at the same site.

If $j < 0$, then we perform induction on the difference $\ell$ in the locations of $r-1$ and $r$, assuming that $r$ is not a singleton and no positions between $r$ and $r-1$ are vacant. If $\ell = 0$, then we are done as argued before. Suppose that $\ell \geq 1$, and for induction suppose that if $r$ is $\ell-1$ positions to the right of $r-1$ with $r$ not a singleton and no vacancies between them, then they will couple. 

Now, for distance $\ell$, the initial situation is
\[
\begin{matrix}
&&  & r & & \\
\dots & r-1 & \underbrace{\dots}_{\ell-1} & a & b & \dots
\end{matrix},
\]
where $a, b \in [n+1]$. We topple the site containing $a$ and $r$, then there are now two sub-cases.

\begin{enumerate}[(1)]

\item
\label{it:r<a}
If $r < a$, then we land in either
\[
\begin{matrix}
& r & & a & \\
\dots & r-1 & \blank & b & \dots
\end{matrix}
\]
if $\ell = 1$, or in
\[
\begin{matrix}
&&& r & & a & \\
\dots & r-1 & \underbrace{\dots}_{\ell-2} & x & \blank & b & \dots
\end{matrix}
\]
if $\ell > 1$ and $x$ is the chip immediately to the left of $a$ initially.
In the former case, $r-1$ and $r$ are already at the same site, and hence coupled. In the latter, we are in a similar situation as what we started with, but $\ell$ has reduced by $1$. Therefore, we are done by the induction assumption.

\item
\label{it:r>a}
Suppose $a < r$ and hence $a < r-1$. Then, in the first step,
we arrive in
\[
\begin{matrix}
&&& a & & r & \\
\dots & r-1 & \underbrace{\dots}_{\ell-2} & x & \blank & b & \dots
\end{matrix},
\]
where again $x$ is the chip immediately to the left of $a$ initially, and $x = r-1$ if $\ell = 1$. Note in particular that \cref{lem:final-move} ensures that when $r$ moves to the right, that is not its final move. Therefore, there must eventually be some $b$ (maybe after toppling some chips to the right of $r$) which lands in the same position as $r$ as shown above.

Now, we only perform topplings on the left half until a particle, say $y$, lands on the same site as $r-1$. Ignoring all the other particles, we then have
\[
\begin{matrix}
&& y && & r & \\
\dots & z & r-1 & \blank & \underbrace{\dots}_{\ell-1} & b & \dots
\end{matrix}.
\]
The key observation is that $y < r-1$. Therefore, at the next stage, we will arrive at
\[
\begin{matrix}
& y &&& & r & \\
\dots & z & \blank & r-1  & \underbrace{\dots}_{\ell-1} & b & \dots
\end{matrix}.
\]
Therefore, we are back in the same situation as before with $\ell$ sites in between, except that both $r-1$ and $r$ are shifted to the right. Now, if $b > r$, we are back in \cref{it:r<a} and we are done by induction. If not, we repeat this argument and end up again with $\ell$ sites between $r-1$ and $r$. Since $r$ will have to move to the left eventually (again, by \cref{lem:final-move}), we will arrive in the situation with \cref{it:r<a} and the result is proved by induction.
\end{enumerate}

The proof of statement (2) is similar---it follows directly by symmetry in the odd case, and by a similar argument to the above in the even case. The proof of (3) follows from \cref{lem:final-move}(3): chips $r = \floor{\frac{n}{2}} + 1$ and $r + 1$ must couple in the final move. 
\end{proof}

We now move towards a structural characterization for toppleable permutations.
To do so, we will find it useful to define the notion of a {\em pass} for a fixed sequence of topplings; recall that the sequence of topplings does not matter by \cref{prop:determ}(1). 

Let $\pi \in S_{n}$. Then, we denote the tuple counting the number of chips at each site of $L_n$ in $\conf{\pi}{r}$ by $\conf{\pi}{r} = (\blank,1,\dots,1,\hat 2,1,\dots,1,\blank)$, where we have marked the origin with a hat. Equivalently, this is the corresponding unlabeled configuration of $\conf{\pi}{r}$. Let us consider what happens to $\conf{p}{r}$ after the first few topplings:
\begin{align*}
\conf{p}{r} 
\to & (\blank,1,\dots,1,1,2, \hat\blank ,2,1,1,\dots,1,\blank) \\
\to & (\blank,1,\dots,1,2,\blank ,\hat 2,\blank ,2,1,\dots,1,\blank)\\
\to & (\blank,1,\dots,2,\blank ,1,\hat 2,1,\blank ,2,\dots,1,\blank).
\end{align*}
At this point, we leave the origin unchanged and start to topple the vertices with 2 chips both on the left and right simultaneously, until we reach the end. We then arrive at the configuration $(1,\blank,1,\dots,1,\hat 2,1,\allowbreak\dots, 1,\blank,1)$. 
Now, the extremal points cannot be modified by any further topplings and are fixed. We call this sequence of topplings the {\em first pass}. This consists of $n$ individual topplings. Similarly, the {\em second pass} will be initiated begin by toppling the origin in a similar way, and we will end up with  $(1,1,\blank,1,\dots,1,\hat 2,1,\dots,1,\allowbreak \blank,1,1)$. We continue this way. 
If $n$ is odd, then we see that after $(n+1)/2$ passes, the configuration will freeze leaving the origin empty.
If $n$ is even, then after $n/2-1$ moves, we end up with
$(1,\dots,1,\blank, \hat 2,1,\blank,1,\dots,1 \allowbreak)$. We then declare the $(n/2)$'th pass to be the one that topples at the origin and site 1, freezing the configuration leaving site 1 empty.

We make two elementary observations about these passing moves. First, every chip between the two vacant sites topples at least once in every pass. Second, if $\pi \in S_{2m+1}$ is $r$-toppleable, then for $1 \leq i \leq m+1$, $i$ and $2m+2-i$ get fixed in their correct positions at the end of the $i$'th pass.
For example, the result of passes on the $\conf{\rho}2$ and $\conf{\sigma}2$ from \eqref{eg-conf} are as follows:
\begin{align*}
&\conf{\rho}2 
\underset{\text{pass}}{\overset{\text{first}}{\longrightarrow}}
\begin{matrix}
&& 2 \\
1 & \blank & 4 & 3 & \blank & 5
\end{matrix} 
\underset{\text{pass}}{\overset{\text{second}}{\longrightarrow}}
\begin{matrix}
1 & 2 & 3 & \blank & 4 & 5
\end{matrix}, \\
&\conf{\sigma}2 
\underset{\text{pass}}{\overset{\text{first}}{\longrightarrow}}
\begin{matrix}
&&& 2 \\
1 & \blank & 3 & 6 & 4 & \blank & 5
\end{matrix} \\
&\underset{\text{pass}}{\overset{\text{second}}{\longrightarrow}}
\begin{matrix}
&&& 3 \\
1 & 2 & \blank & 4 & \blank & 6 & 5
\end{matrix}
\underset{\text{pass}}{\overset{\text{third}}{\longrightarrow}}
\begin{matrix}
\\
1 & 2 & 3 & \blank & 4 & 6 & 5
\end{matrix}.
\end{align*}

\begin{lem}
\label{lem:pos1}
If $\pi \in S_{2m+1}$ is toppleable, then $1$ is in position at most $m+1$ in $\pi$. Conversely, if $1$ (resp $n$) is in position at most $m+1$ (at least $m+1$) in $\pi$, then $1$ (resp $n+1$) is in the first
(resp. last) position in $\mathcal{T}(\conf{\pi}{m+1})$.

If $\pi \in S_{2m}$ is toppleable, then $1$ is in position at most  $m$  in $\pi$. Conversely, if $1$ (resp. $2m$) is in position at most $m$ (resp. at least $m$) in $\pi$, then $1$ (resp. $2m + 1$) is in the first (resp. last) position in $\mathcal{T}(\pi^{(m + 1)})$. 
\end{lem}

\begin{proof}
By \cref{thm:certi}, it suffices to consider $\conf{\pi}{\floor{n/2}+1}$.
Suppose $1$ is 
to the right of the origin in $\conf{\pi}{\floor{n/2}+1}$. Then, in the first pass, $1$ will move exactly one position to the left (since it is smallest) at the end of the first pass. Therefore, $1$ is not frozen in its correct position, which is the extreme left. So $\pi$ cannot be toppleable.

Conversely, suppose $1$ is in a position on or to the left of the origin in $\conf{\pi}{\floor{n/2}+1}$. Then it gets a partner at some point during the first pass. After that time, it keeps moving left for all future times until the first pass ends and gets placed at the extreme left, its correct position. A similar argument works for $n$, completing the proof.
\end{proof}

We are now in a position to characterize toppleable permutations.
This characterization involves bounds on the difference between values and positions and is in the spirit of so-called {\em Vesztergombi permutations} \cite{vesztergombi-1974,lovasz-vesztergombi-1978}, where these differences have global bounds.

\begin{thm}
\label{thm:topp-struct}
A permutation $\pi \in S_{2m+1}$ is $(m+1)$-toppleable if and only if 
\begin{align*}
 \pi_i \leq  m+i,& \quad 1 \leq i \leq m, \\
 \pi_i \geq  i-m,& \quad m+1 \leq i \leq 2m+1. 
\end{align*}
Equivalently, 
\begin{align*}
\pi^{-1}_i \in \{1,\dots,m+i\}, & \quad 1 \leq i \leq m+1,\\
\pi^{-1}_i \in \{i-m, \dots, 2m+1\},&  \quad m+2 \leq i \leq 2m+1.
\end{align*}

A permutation $\pi \in S_{2m}$ is 
$(m+1)$-toppleable if and only if 
\begin{align*}
 \pi_i \leq  m+i,& \quad 1 \leq i \leq m, \\
 \pi_i \geq  i-m+1,& \quad m+1 \leq i \leq 2m. 
\end{align*}
Equivalently, 
\begin{align*}
\pi^{-1}_i \in \{1,\dots,m+i-1\}, & \quad 1 \leq i \leq m,\\
\pi^{-1}_i \in \{i-m, \dots, 2m\},&  \quad m+1 \leq i \leq 2m.
\end{align*}
\end{thm}

\begin{proof}
We will prove the forward implication first. 
Suppose for some $1 \leq i \leq m$, $i$ is in a position greater than $i+m$ in $\pi$. Therefore, it is to the right of the origin in $\conf{\pi}{m+1}$. 
If $i=1$, we are done by \cref{lem:pos1}. If not, consider the situation after $i-1$ passes. After each pass, $i$ either moves left by exactly one position or moves to the right by an arbitrary number of positions. Therefore, at the end of the $(i-1)$'th pass, $i$ is strictly to the right of the origin. Therefore, arguing exactly as in \cref{lem:pos1}, $i$ cannot be in its correct position at the end of the $i$'th pass. So, $\pi$ cannot be $(m+1)$-toppleable.
The case when $i \geq m+2$ is in a position to the left of $i-m$ in $\pi$ is done by symmetry.

For the converse, suppose $i \in \{\pi_1,\dots,\pi_{m+i}\}$ for $1 \leq i \leq m+1$. Then in $\conf{\pi}{m+1}$, $i$ lies to the left of position $i$ in $L_n$. Denote the position of $i$ in $L_n$ after the $j$'th pass by $p_j(i)$. We will show for every $j$: after the $j$'th pass, we have $i = 1, \ldots, j$ fixed in their correct positions, and for each $i = j + 1, \ldots, m + 1$, $p_j(i) < i - j$ (Notice this means that every $i \leq m + 1$ is fixed correctly after the $(m + 1)$'th pass).

We will prove this by double induction, first on $j$: the base case is simply the $0$'th pass, which is satisfied by our assumptions. Then supposing our statement holds for pass $j - 1$, consider pass $j$. Notice first that for $i = j$, we have $p_{j-1}(j) \leq 0$ by the induction hypothesis. Since $j$ is the smallest non-fixed chip, then by the same reasoning as Lemma \ref{lem:pos1}, in the $j$'th pass, $j$ topples left until it is fixed into its correct position (for $j=1$, this is exactly Lemma \ref{lem:pos1}). Thus, after pass $j$, chips $1, \ldots, j$ are correctly fixed.

Now, for chips $i= j + 1,\ldots, m + 1$, we will use induction on $i$ to argue that during the $j$'th pass, $i$ never moves to the right of position $i - j$, and must end strictly left of position $i - j$ (i.e. $p_j(i) < i-j$ as desired). The base case is $i = j$, and simply follows from our previous argument about $i = j$. Now, supposing the statement holds for $j, \ldots, i - 1$, recall from our (outer) induction hypothesis that $p_{j-1}(i) \leq i - j$. We consider two cases:
\begin{enumerate}
    \item If $p_{j-1}(i) < 0$, then in pass $j$ it can move at most one space to the right, so it will fall in a desired position (left of $i - j$).
    
    \item If $p_{j-1}(i) \geq 0$, then it may move to the right multiple times. However, it either never reaches position $i-j$, in which case it is in a desired position, or it reaches (or starts at) position $i-j$. But upon reaching this position, it must topple left, as the (inner) induction hypothesis gives that no smaller chip can reach that position. Thus, in either case, it ends at a position left of $i-j$, so $p_j(i) < i=j$ as desired.
\end{enumerate}
When $\pi^{-1}_i \in \{i - m, \ldots, 2m+1\}$ for $m + 2 \leq i \leq 2m +1$, an identical argument shows that each of these chips lies in its correct position after $m + 1$ passes. The argument for even $n$ is mostly identical and omitted. 
\end{proof}

We now prove a bijective correspondence relating $(m+1)$-toppleable permutations to permutations with the correct excedance set,

\begin{lem}
\label{lem:bij}
Permutations $\pi \in S_{2m+1}$ such that $\pi_i \leq m+i$ for $1 \leq i \leq m$ and $\pi_i \geq i-m$ for $m+1 \leq i \leq 2m+1$
are in bijection with permutations in $S_{2m+1}$ whose excedance set is $\{1,\dots,m\}$.

Permutations $\pi \in S_{2m}$ such that $\pi_i \leq m + i$ for $1 \leq i \leq m$ and $\pi_i \geq i-m + 1$ for $n + 1 \leq i \leq 2m$ are in bijection with permutations in $S_{2m}$ whose excedance set is $\{1,\dots,m-1\}$.
\end{lem}

\begin{proof}
We consider the odd case first. Suppose $\pi \in S_{2m+1}$ satisfies the above conditions. 
Then we define $\sigma \in S_{2m+1}$ as 
\[
\sigma_i = \begin{cases}
2m+2 - \pi_{m+1-i} & 1\le i \le m, \\
2m+2 - \pi_{3m+2-i} & m+1\le i \le 2m+1\,.
\end{cases}
\]
Then we claim that $\sigma$ has excedance set $\{1,\dots,m\}$.
First, suppose $\sigma$ has the desired excedance set. By definition,
$\sigma_i > i$ for $1 \leq i \leq m$. Thus, $\pi_{m+1-i} < 2m+2-i$, or equivalently, $\pi_i < m+i+1$.
Also, $\sigma_i \leq i$ for $m+1 \leq i \leq 2m+1$. Therefore, $\pi_{3m+2-i} \geq 2m+2-i$ for this range of $i$.
Equivalently, $\pi_i \geq i-m$. Thus, $\pi$ satisfies the above conditions.

Conversely, suppose $\pi$ satisfies the above conditions. There are three cases to consider. First, 
consider the entries $i$ such that $1 \leq i \leq m+1$. Then $i \in \{\pi_1,\dots,\pi_{m+i}\}$.
Thus, $2m+2-i \in \{ \sigma_1, \dots, \sigma_m \} \cup \{ \sigma_{2m+2-i}, \dots, \sigma_{2m+1} \}$, where 
$m+1 \leq 2m+2-i $. Now, if $i$ belonged to the set $\{\pi_1,\dots,\pi_{m}\}$, then
$2m+2-i \in \{\sigma_1,\dots,\sigma_m\}$ and we get an excedance in $\{1,\dots,m\}$.
If not, $2m+2-i \in \{\sigma_{2m+2-i}, \dots, \sigma_{2m+1}\}$, and we do not get an excedance in 
$\{2m+2-i, \dots, 2m+1\}$.

Second, consider the entries $i$ such that $m+2 \leq i \leq 2m$. Then $i \in \{\pi_{i-m}, \dots, \allowbreak \pi_{2m+1}\}$.
Therefore, $2m+2-i \in \{ \sigma_1, \dots, \sigma_{2m+1-i}\} \cup \{ \sigma_{m+1}, \dots, \sigma_{2m+1} \}$,
where $2 \leq 2m+2-i \leq m$. Now, if $i$ belonged to the set $\{\pi_{i-m},\dots,\pi_{m}\}$, 
then $2m+2-i \in \{\sigma_1,\dots,\sigma_{2m+1-i}\}$ and we get an excedance in $\{1,\dots,2m+1-i\}
\subset \{1, \dots, m\}$. If not, $2m+2-i \in \{\sigma_{m+1}, \dots, \sigma_{2m+1}\}$, and we do not get an excedance in 
$\{m+1, \dots, 2m+1\}$.

Lastly, the entry $2m+1 \in \{\pi_{m+1},\dots, \pi_{2m+1} \}$. Thus, $1 \in \{\sigma_{m+1}, \dots, \allowbreak \sigma_{2m+1} \}$
and the value $1$ can never contribute an excedance. Thus, we do not get an excedance in 
$\{m+1, \dots, 2m+1\}$. Therefore, every chip which lands in positions $\{1,\dots,m\}$ of $\sigma$
gives an excedance and which lands outside it does not, proving that $\sigma$ has excedance set exactly $\{1,\dots,m\}$.

For the even case, suppose $\pi \in S_{2m}$ satisfies the conditions above. Then we define $\sigma \in S_{2m}$ as 
\[
\sigma_i = \begin{cases}
2m+1 - \pi_{m-i} & 1 \leq i \leq m-1, \\
2m+1 - \pi_{3m-i} & n \leq i \leq 2m.
\end{cases}
\]
The overall strategy of proof is similar to the odd case.
\end{proof}

\begin{eg}
As an illustration of \cref{lem:bij} in the odd case, let $\pi = 31524 \in S_5$ which satisfies $\pi_i \leq 2+i$ for $1 \leq i \leq 2$ and $\pi_i \geq i-2$ for $3 \leq i \leq 5$. Then $\sigma = 53241$, which has excedance set $\{1,2\}$.

For the even case, let $\pi = 216435 \in S_6$ which satisfies $\pi_i \leq 3+i$ for $1 \leq i \leq 3$ and $\pi_i \geq i-2$ for $4 \leq i \leq 6$. Then $\sigma = 652431$, which has excedance set $\{1,2\}$.
\end{eg}

We are now in a position to prove our main result.

\begin{proof}[Proof of \cref{thm:topp=exc}]
By \cref{thm:certi}, we see that $\pi \in S_n$ is toppleable if it is $(\floor{n/2}+1)$-toppleable.
According to \cref{thm:topp-struct}, $\pi_i \leq \floor{n/2} +i$ for $1 \leq i \leq \floor{n/2}$ and $\pi_i \geq i- \ceil{n/2}+1$ for $\floor{n/2}+1 \leq i \leq n$.
Now, \cref{lem:bij} proves that the number of such permutations is $a_{n,\floor{(n-1)/2}}$ bijectively, completing the proof.
\end{proof}

\section{Collapsed permutations}
\label{sec:collap}

As explained before, Hopkins, McConville and Propp considered toppling on $\mathbb{Z}$ starting with $n$ chips at the origin. One of their main results~\cite[Theorem 13]{hopkins2016sorting} is that when $n$ is even, the final configuration is always $\id$, the identity permutation. Along the way, they prove bounds on the possible positions of chip $k$ at every step of the toppling process. Specifically, they show~\cite[Lemma 12]{hopkins2016sorting} that the position of chip $k$ lies between $- \floor{(n+1-k)/2}$ and $\floor{k/2}$.
When $n$ is odd (resp. even), $n = 2m+1$ (resp. $n = 2m$), the final configuration will contain single chips in all positions (resp. all positions but the origin)  $-m$ through $m$. 
We now apply this condition to count possible permutations arising from this condition switching notation to considering permutations as bijections from the set $[n]$ to itself.

\begin{defn}
\label{def:Gn}
We say that a permutation $\pi \in S_n$ is collapsed if
\[
\pi^{-1}_k \geq  
\begin{cases}
\ceil{k/2} & n \text{ odd}, \\
1 + \floor{k/2} & n \text{ even}
\end{cases} 
\quad \text{and} \quad
\pi^{-1}_k \leq \ceil{n/2} + \floor{k/2}.
\]
Let $G_n$ be the subset of collapsed permutations in $S_n$.
\end{defn}

For example, $G_3 = \{123, 132, 213\}$ and $G_4 = \{1234, 1324\}$.
Since, for $n$ even, the only permutation that appears as a result of toppling is $\id$, counting the cardinality of $G_n$ in that case is not directly relevant to the toppling problem.

To state our results, we recall a well-known combinatorial triangle.
The \emph{Seidel triangle} is the triangular sequence $S_{n,k}$ for $n \geq 1$ given by
\begin{align*}
S_{1,1} =& 1, \\
S_{n,k} =& 0, \quad k < 2 \text{ or }  (n+3)/2 < k, \\
S_{2n,k} =& \ds \sum_{i \geq k} S_{2n-1,i}, \\
S_{2n+1,k} =& \ds \sum_{i \leq k} S_{2n,i}.
\end{align*}
The first few rows of the Seidel triangle are given in \cref{tab:seidel}.

\begin{center}
\begin{table}[htbp!]
\begin{tabular}{c|ccccc}
$n \backslash k$ & 2 & 3 & 4 & 5 & 6 \\
\hline
1 & 1\\
2 & 1\\
3 & 1 & 1 \\
4 & 2 & 1 \\
5 & 2 & 3 & 3\\
6 & 8 & 6 & 3\\
7 & 8 & 14 & 17 & 17 \\
8 & 56 & 48 & 34 & 17 \\
9 & 56 & 104 & 138 & 155 & 155\\
10 & 608 & 552 & 448 & 310 & 155
\end{tabular}
\caption{Rows $1$ through $10$ of the Seidel triangle.}
\label{tab:seidel}
\end{table}
\end{center}

The (unsigned) {\em Genocchi numbers of the first kind} $g_{2n}$, $n \geq 1$ are the numbers on the rightmost diagonal of the Seidel triangle and they count the number of permutations in $S_{2n-3}$ whose excedance set is $\{1,3,\dots,2n-5\}$. The exponential generating function of $g_{2n}$ is given by
\[
\sum_{n \geq 0} g_{2n} \frac{x^{2n}}{(2n)!} = x \tan \left( \frac{x}{2} \right).
\]
The first few numbers of this sequence are~\cite[A110501]{OEIS}
\[
\{g_{2n}\}_{n=1}^8 = \{1, 1, 3, 17, 155, 2073, 38227, 929569\}.
\]

\begin{thm}
\label{thm:genocchi1}
The cardinality of $G_{2n+1}$ is $g_{2n+4}$.
\end{thm}

\begin{proof}
Suppose $\pi \in G_{2n+1}$. Then one can check that the definition of $G_n$ in \cref{def:Gn} is equivalent to saying $\pi_i \leq 2i$ and $\pi_{n+1+i} \geq 2i$ for $1 \leq i \leq n$.
Define a map $f: G_{2n+1} \to S_{2n+1}$ which send $\pi \mapsto \sigma = (\sigma_1,\dots,\sigma_{2n+1})$ such that $\sigma_{2i} = \pi_i$
and $\sigma_{2i-1} = \pi_{n+1+i}$ for $1 \leq i \leq n$, and $\sigma_{2n+1} = \pi_{n+1}$. By the previous sentence, it is immediate that the excedance set of $\sigma$ is $\{1,3,\dots,2n-3,2n-1\}$.
Since $f$ is injective, one can check that the inverse map is well-defined and that $f^{-1}(\sigma) \in G_{2n+1}$ for $\sigma \in S_{2n+1}$ with excedance set $\{1,3,\dots,2n-3,2n-1\}$ in a similar way.
\end{proof}

As an example, the bijection for $n=1$ is illustrated below:
\[
\begin{array}{|c|c|}
\hline
G_3 & \pi \in S_3 \text{ with excedance set } \{1\} \\
\hline
132 & 213 \\
123 & 312 \\
213 & 321 \\
\hline
\end{array}
\]

The {\em median Genocchi numbers} or the {\em Genocchi numbers of the second kind}, denoted $H_{2n+1}$ are given by the leftmost diagonal of the Seidel triangle (see \cref{tab:seidel}). They count among other things, ordered pairs $((a_1,\dots,a_{n-1}), \allowbreak (b_1,\dots,b_{n-1})) \in \mathbb{Z}^{n-1} \times \mathbb{Z}^{n-1}$ such that $0 \leq a_k \leq k$ and $1 \leq b_k \leq k$ for all $k$ and $\{a_1,\dots,a_{n-1}, \allowbreak b_1,\dots,b_{n-1}\} = [n-1]$~\cite{hetyei-2019}.
The first few numbers of this sequence are~\cite[A005439]{OEIS}
\[
\{H_{2n-1}\}_{n=1}^8 = \{1, 2, 8, 56, 608, 9440, 198272, 5410688\}.
\]
No particularly simple formula or generating function seems to be known for $H_{2n+1}$. In terms of the Genocchi numbers of the first kind, we have~\cite{han-zeng-1999}
\[
H_{2n+1} = \sum_{i=0}^n g_{2n-2i} \binom{n}{2i+1}.
\]
Although it is not clear either from the above definition or the formula, it is true that $H_{2n-1}$ is divisible by $2^{n-1}$. 
The numbers $h_n = H_{2n+1}/2^{n}$ are called the {\em normalized median Genocchi numbers}.
The first few numbers of this sequence are~\cite[A000366]{OEIS}
\[
\{h_{n}\}_{n=0}^7 = \{1, 1, 2, 7, 38, 295, 3098, 42271\}.
\]
A classical combinatorial interpretation for these are the Dellac configurations, which we now define.

A {\em Dellac configuration} of order $n$ is a $2n \times n$ array containing $2n$ points, such that every row has a point, every column has two points, and the points in column $j$ lie between rows $j$ and $n+j$, both inclusive, $1 \leq j \leq n$.
For example, the two Dellac configurations of order $2$ are
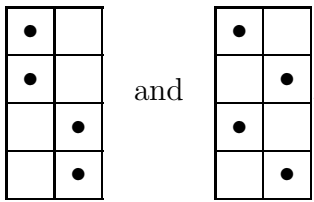
\[
\begin{ytableau}
\bullet & \\
\bullet & \\
& \bullet \\
& \bullet \\
\end{ytableau}
\quad \raisebox{-0.6cm}{\text{and}} \quad
\begin{ytableau}
\bullet & \\
& \bullet \\
\bullet & \\
& \bullet \\
\end{ytableau}\raisebox{-0.6cm}{.}
\]
Dellac's result~\cite{dellac-1900} is that the number of Dellac configurations of order $n$ is $h_n$.

\begin{thm}
\label{thm:genocchi2}
The cardinality of $G_{2n}$ is given by $H_{2n-1}$.
\end{thm}

\begin{proof}
Suppose $\pi \in G_{2n}$. Directly from the definition of $G_n$ in \cref{def:Gn}, we see that both $2i$ and $2i+1$ have to lie in positions between $i+1$ and $i+n$, both inclusive, for $1 \leq i \leq n-1$. Therefore, interchanging $2i$ and $2i+1$ will also give a permutation in $G_{2n}$, and this can be done independently for each $i$. Thus, $\# G_{2n}$ is divisible by $2^{n-1}$.

Now, we focus on $\pi \in G_{2n}$ such that 
$\pi^{-1}_{2i} < \pi^{-1}_{2i+1}$ for $1 \leq i \leq n-1$, i.e. those for which $2i$ precedes $2i+1$ in one-line notation. We will now construct a bijection between such permutations and Dellac configurations of order $n-1$.
Again, from the definition of $G_n$ in \cref{def:Gn}, we have that 
$\pi_1 = 1$,
$2 \leq \pi_i \leq 2i-1$ for $1 \leq i \leq n$,
$2(i-n) \leq \pi_{i} \leq 2n-1$ for $n+1 \leq i \leq 2n-1$,
and $\pi_{2n} = 2n$.
Since the first and last entries are forced, we focus on $(\pi_2,\dots,\pi_{2n-1})$. 
Construct a configuration $C$ of points on an $(2n-2) \times (n-1)$ array as follows: for $2 \leq i \leq 2n-1$, place a point in position $(i-1, \floor{\pi_i/2})$. We claim that $C$ is a Dellac configuration. First of all, it is obvious that there is one point in every row and two points in every column. Now, if $\floor{\pi_i/2} = j$, $\pi_i$ is either $2j$ and $2j+1$ and in both cases, $j \leq i-1 \leq j+n-1$ (from the first paragraph of the proof). This is precisely the condition for $C$ to be a Dellac configuration of order $n-1$.
For example, the permutation $(1, 2, 4, 3, 6, 5, 7, 8)$ is in bijection with the Dellac configuration
\[
\begin{ytableau}
\bullet & & \\
 & \bullet & \\
\bullet  & & \\
 & & \bullet \\
 & \bullet & \\
 & & \bullet \\
\end{ytableau}
\]

Conversely, given a Dellac configuration $C$ of order $n-1$, label the points in column $j$ as $2j$ and $2j+1$ from top to bottom for all $1 \leq j \leq n-1$. Now, read the labels in each row from top to bottom to form the tuple $\pi = (\pi_2,\dots,\pi_{2n-1})$. By construction, this is a permutation of the set $\{2,\dots,2n-1\}$. Because of the Dellac constraint, $2j$ and $2j+1$ lie between positions $j$ and $n-1+j$ in $\pi$ for each $j$. Prepending $1$ and appending $2n$ to $\pi$ therefore gives a permutation in $G_{2n}$, completing the proof.
\end{proof}

\begin{rem}
Even though several combinatorial interpretations are known for the median Genocchi numbers $H_{2n-1}$, it seems like $G_{2n}$ is one of the few for which the divisibility by $2^{n-1}$ is manifest.
\end{rem}

\section{AUSOs of complete bipartite graphs and excedances}
\label{sec:complete-bipartite}

Recall the definition of excedance, acyclic orientation (AO) and acyclic unique sink orientation (AUSO) from \cref{sec:topp}.
We denote the set of AUSOs of a graph $G$ for which the unique sink is always the fixed vertex $s$ as $\mathcal{A}(G, s)$. 
Recall also from \cref{sec:topp} that the number of AUSOs with fixed sink is independent of the choice of sink. We therefore pay no mind to the actual choice of vertex $s$ in the following sections.
In this section, we will prove \cref{thm:topp=ursell} bijectively.
The discussion will be made easier by considering \textit{topological sorts} of the vertices of an acyclic orientation.

\begin{defn}
Given a directed acyclic graph (e.g., an AO of any graph), a \emph{topological sort} of the vertices is a total ordering $\tau$ on the vertices such that for every directed edge $e = (u, v)$, then $\tau(u) < \tau(v)$. We will often say that an \textit{undirected} graph $G$ has topological sort $\tau$ if there is an AO on $G$ with topological sort $\tau$. 
\end{defn}

\begin{eg}
Notice that a given directed acyclic graph can have multiple topological sorts. For example, 
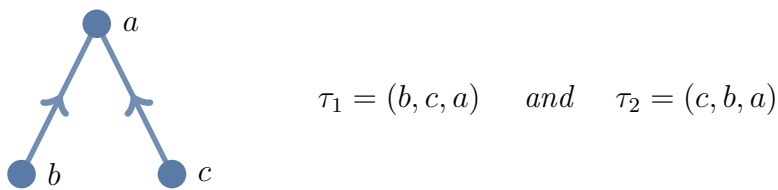
\begin{figure}[h!]
    \centering
    \begin{tikzpicture}[scale=2]
    \begin{scope}[thick] 
        \draw[ro] (0, 0)--(-0.5, -1);
        \draw[ro] (0, 0)--(0.5, -1);
    
        \node[main, label=right:$a$] at (0, 0) {};
        \node[main, label=right:$b$] at (-0.5, -1) {};
        \node[main, label=right:$c$] at (0.5, -1) {};

        \node[] at (3, -0.5) {$\tau_1 = (b,c,a) \quad \text{ and } \quad \tau_2 = (c,b,a)$};
    \end{scope}
    \end{tikzpicture}
    \caption{Directed acyclic graph with two different topological sorts $\tau_1$ and $\tau_2$.}
\end{figure}

\end{eg}

In a directed acyclic graph, we call two vertices \textit{incomparable} if there is no directed path from one vertex to the other. Likewise, in a topological sort of an undirected graph $G$, two vertices are incomparable if they are incomparable in the AO induced by the sort. 
A topological sort of $G$ in which every pair of vertices that are incomparable appear in increasing numerical order is called a \emph{canonical topological sort} of $G$.

\begin{lem}
\label{lem:canonical-sorts}
If $K$ is a complete multipartite graph with vertices labeled $1, 2, \ldots, n$, then there is a bijection between AOs and canonical topological sorts of $K$. 
\end{lem}

\begin{proof}
It is immediate that any canonical topological sort of the vertices of 
any graph $G$ uniquely induces an AO on $G$.

Conversely, given an AO $\mathcal{O}$ of $K$, it has a canonical preimage. Two vertices are incomparable in $\mathcal{O}$ if and only if they are in the same vertex set of $K$, and have identical in/out-neighborhoods (otherwise there would be a directed path from one to the other). Thus, as $K$ is complete multipartite, incomparibility is an equivalence relation, and any topological sort of $\mathcal{O}$ must order the equivalence classes with respect to one another. Any ordering of the vertices within each class is still a sort of $\mathcal{O}$, and so simply ordering each class in increasing order gives a canonical topological sort preimage of $\mathcal{O}$.
\end{proof}

For $0 \leq m \leq n-1$, denote by $\mathcal{E}(n, m)$ the set of permutations in $S_n$ with excedance set exactly $\{1,\ldots,m\}$. That is, denoting the excedance set by $\ex(\sigma)$, 
\[
    \mathcal{E}(n, m) = \bigl\{\sigma \in S_{n} \bigm| \ex(\sigma) = \{1, \ldots, m\} \bigr\}.
\]

It will be convenient to consider the (disjoint) cycle decomposition of a permutation $\sigma = \sigma_1 \cdots \sigma_{k}$. It is not hard to see that $\sigma \in \mathcal{E}(m + n, m)$ if and only if each individual cycle $\sigma_i = (a_1 ,\ldots, a_j)$ satisfies that for each $k$,
\begin{itemize}
    \item $a_k < a_{k + 1}$ if and only if $a_k \in \{1,\ldots, m\}$ (since $\sigma(a_k) = a_{k + 1} > a_k$).
    \item $a_k \geq a_{k + 1}$ if and only if $a_k \in \{m + 1, \ldots, m + n\}$.
\end{itemize} 
(where we let $a_{j + 1} = a$). For example, if $m = 4$, $n = 4$, 
then we may consider the permutation $\sigma = (27)(138645)$. Both cycles follow the the above property: $(27)$ increases only after 2, and $(138645)$ increases after 1, 3, and 4. And indeed, $\ex(\sigma) = \{1, 2, 3, 4\}$. 

While discussing the complete bipartite graph $K_{m, n}$, we refer to the `Left' vertex set containing $m$ vertices as $L = \{1, \ldots, m\}$, and the `Right' containing $n$ vertices as ${R = \{m + 1, \ldots, m + n\}}$. 

Denote by $\mathcal{R}(m, n)$ the set of AOs of $K_{m, n}$ in which there are no sinks in $L$ (sinks only allowed in $R$). By simply adding a vertex to $L$ to get $K_{m + 1, n}$ and demanding this vertex be a sink, we maintain the acyclic property, and the added vertex becomes the unique sink. Removing the vertex gives the inverse, and the resulting bijection shows 
\[
    \abs{\mathcal{R}(m, n)} = \abs{\mathcal{A}(K_{m + 1, n}, s)}.
\]

We will give a bijection between AUSOs $\mathcal{A}(K_{m + 1, n}, s)$ and permutations with a fixed excedance set $\mathcal{E}(m + n, m)$. We use the discussion above to write the bijection in terms of the cycle decomposition of permutations in $\mathcal{E}(m + n, m)$ and orientations in $\mathcal{R}(m, n)$. 

Let $f\colon \mathcal{E}(m + n, m) \to \mathcal{R}(m, n)$ be defined as follows. For $\sigma \in \mathcal{E}(m + n, m)$, construct $f(\sigma)$ by:
\begin{enumerate}
    \item Write each cycle in the cycle decomposition of $\sigma$ such that its least element appears first.
    
    \item Order the cycles relative to each other from right to left by least element (disjoint cycles commute). 
    
    \item Remove the parentheses around the cycle decomposition as written in step (2), and consider the resulting sequence of vertices as a topological sort inducing an orientation on $K_{m, n}$.
\end{enumerate}

\begin{eg}
Suppose $m = 4$, $n = 5$, and we are given the permutation 
\[
    \sigma = \begin{pmatrix} 1 & 2 & 3 & 4 & 5 & 6 & 7 & 8 & 9\\3 & 9 & 6 & 7 & 5 & 2 & 4 & 8 & 1 \end{pmatrix} = (13629)(47)(5)(8).
\]
It is easy to verify that $\sigma \in \mathcal{E}(9, 4)$.
To construct $f(\sigma)$ notice that each cycle is already written according to step 1. Step 2 yields the decomposition $(\underline{8})(\underline{5})(\underline{4}7)(\underline{1}3629)$. Finally, step 3 gives the topological sort $8,5,4,7,1,3,6,2,9$ of $K_{4,5}$ with the resulting orientation being acyclic and containing no sink in $L = \{1, 2, 3, 4\}$. We can check this by seeing that the last vertex in the topological sort is in $R$, so no vertex in $L$ can be a sink.
\end{eg}

\begin{lem}
The function $f$ is well-defined.
\end{lem}

\begin{proof}
Consider any permutation $\sigma$. Since $f(\sigma)$ is an orientation determined by a topological sort, it must be acyclic. 
To check that $f$ indeed gives an AO with no sink in $L$, write $\sigma$ according to step 2 of the description of $f$ as $\sigma=\sigma_1 \cdots  \sigma_k$. The rightmost vertex $v$ in the topological sort $f(\sigma)$ is the final term in the cycle $\sigma_k = (\underline{1},\ldots, v)$. In particular, $\sigma(v) = 1$ and since 1 is the least element, $\sigma(v) \leq v$. Therefore $v \not \in \ex(\sigma)$, so
\[
    v \in (L \cup R) \setminus \ex(\sigma) = \{m + 1, \ldots, m + n\} = R\,.
\]
Since $v$ is the rightmost term of the topological sort, it is a sink. Since it is in $R$, it has all vertices in $L$ as neighbors, and so no vertex in $L$ can be a sink.
\end{proof}

\begin{proof}[Proof of \cref{thm:topp=ursell}]
We now prove that $f$ is a bijection by constructing its inverse. 
Given an orientation $\mathcal{O} \in \mathcal{R}(m, n)$, construct $f^{-1}(\mathcal{O})$ as follows:
 
Since $\mathcal{O}$ is acyclic, it has a topological sort. But it may not be unique, as there could be vertices $u, v$ either both in $L$ or both in $R$ whose positions could be swapped (they are incomparable). We designate a unique topological sort by specifying that whenever $u < v$, with $u$ and $v$ incomparable.
\begin{itemize}
    \item if $u, v \in L$, then $u$ appears before $v$ in the topological sort. 
    \item if $u, v \in R$, then $u$ appears after $v$ in the topological sort.
\end{itemize}

This defines a unique topological sort $v_1, \ldots, v_{m + n}$, and there is a unique way to insert parentheses into the sequence to make it into a cycle decomposition written as specified in step 2 of the description of $f$. In particular, the rightmost cycle contains 1 and all elements to its right, the next cycle contains the next smallest unused element and all unused elements to its right, etc. 

To see that the resulting permutation $\sigma = f^{-1}(\mathcal{O})$ has excedance set $L$, notice that our tie-breaking strategy and the fact that all elements of $L$ are less than those of $R$ ensures the following conditions hold: 
\begin{enumerate}
    \item In the topological sort, $v_i < v_{i + 1}$ if and only if $v_i \in L$;
    \item Because we inserted parentheses by choosing the next smallest unused element, every cycle $(v_i\ldots v_{i + k})$ must have either $k > 0$, $v_i \in L$ and $v_{i + k} > v_i$, or $v_i \in R$ and $v_{i + k} = v_i$ (that is, $k = 0$).
\end{enumerate} 
These two conditions ensure that in every cycle $(v_i, \ldots, v_{i + k})$, we have $v_j < v_{j + 1}$ if and only if $v_j \in L$. Thus, noting that $\sigma(v_j) = v_{j + 1}$, the excedance set of the permutation is $L$, as desired. 
\end{proof}

\begin{eg}
As an example of $f^{-1}$, we reverse the example we did while defining $f$ (recall $m = 4$, $n = 5$). We start with an orientation which has topological sort $5,8,4,7,3,1,6,2,9$ (notice that there is no sink in $L$) as in \cref{fig:AO-eg}.

Notice that in this orientation, vertices $5, 8 \in R$ are incomparable, as well as vertices $1, 3 \in L$. So, using the tie-breaking strategy, we get the unique topological sort: $8,5,4,7,1,3,6,2,9$.
Now, we insert parentheses first left of 1, then left of 4, then 5, then 8 to obtain $f^{-1}(\mathcal{O}) = (8)(5)(47)(13629)$.
Notice that for this permutation, $f^{-1}(f(\sigma)) = \sigma$, as we expected. 
\end{eg}

\begin{figure}[h!]
\centering
\begin{tikzpicture}[scale=2]
\begin{scope}[thick] 
    \draw[ro=0.78] (0, 0)--(-0.5, -1.5); 
    \draw[ro=0.78] (1, 0)--(-0.5, -1.5); 
    \draw[ro=0.78] (2, 0)--(-0.5, -1.5); 
    \draw[ro=0.8] (3, 0)--(-0.5, -1.5); 

    \draw[ro=0.8] (0, 0)--(2.5, -1.5); 
    \draw[ro=0.8] (1, 0)--(2.5, -1.5); 
    \draw[ro=0.8] (2, 0)--(2.5, -1.5); 
    \draw[ro=0.82] (3, 0)--(2.5, -1.5); 

    \draw[fo=0.24] (3, 0)--(0.5, -1.5);
    \draw[fo=0.275] (3, 0)--(1.5, -1.5);
    \draw[fo=0.28] (3, 0)--(3.5, -1.5);

    \draw[ro=0.82] (0, 0)--(1.5, -1.5);
    \draw[ro=0.8] (1, 0)--(1.5, -1.5);
    \draw[ro=0.8] (2, 0)--(1.5, -1.5);

    \draw[fo=0.3] (0, 0)--(0.5, -1.5);
    \draw[fo=0.19] (0, 0)--(3.5, -1.5);

    \draw[fo=0.24] (2, 0)--(0.5, -1.5);
    \draw[fo=0.19] (2, 0)--(3.5, -1.5);

    \draw[ro=0.8] (1, 0)--(0.5, -1.5);
    
    \draw[fo=0.19] (1, 0)--(3.5, -1.5); 

    \node[main, label=right:$1$] at (0, 0) {};
    \node[main, label=right:$2$] at (1, 0) {};
    \node[main, label=right:$3$] at (2, 0) {};
    \node[main, label=right:$4$] at (3, 0) {};
    \node[main, label=right:$5$] at (-0.5, -1.5) {};
    \node[main, label=right:$6$] at (0.5, -1.5) {};
    \node[main, label=right:$7$] at (1.5, -1.5) {};
    \node[main, label=right:$8$] at (2.5, -1.5) {};
    \node[main, label=right:$9$] at (3.5, -1.5) {};
\end{scope}
\end{tikzpicture}
\caption{The AO $\mathcal{O}$ on $K_{4, 5}$. The topological sort specified in the description of $f^{-1}$ is $(8,5,4,7,1,3,6,2,9)$, and this produces the permutation $(8)(5)(47)(13629)$.}
\label{fig:AO-eg}
\end{figure}
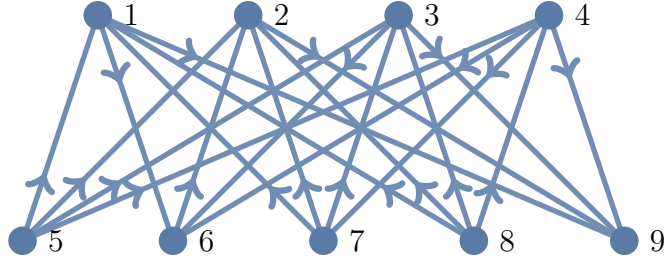

\section{Enumeration of AOs and AUSOs of complete multipartite graphs}
\label{sec:ao-enum}

In this section, we give formulas for the number of acyclic orientations and acyclic orientations with unique sink of complete multipartite graphs. 
We first begin with the easier case of complete bipartite graphs.

\subsection{Complete Bipartite Graphs}

We know from~\cite{ehrenborg-steingrimsson-2000}, that 
\begin{align*}
    \abs{\mathcal{E}(m + n, m)} &= \sum_{i = 1}^{m + 1} (-1)^{m + 1 - i} \cdot i!\cdot i^{n-1}\cdot  \stirling{m + 1}{i}\\
    &= \sum_{i = 1}^{n} (-1)^{n - i} \cdot i!\cdot i^{m}\cdot  \stirling{n}{i} \\
    &= n!n^m - \sum_{j = 1}^{n-1} (-1)^{j-1} \cdot (n-j)!\cdot (n-j)^{m}\cdot  \stirling{n}{n-j},
\end{align*}
where $\stirling{n}{k}$ is the Stirling number of the second kind, 
the second equality follows from the symmetry $\mathcal{E}(m + n, m) = \mathcal{E}(m + n, n-1)$~\cite{ehrenborg-steingrimsson-2000} (which itself isn't immediately obvious), and the third follows by pulling out the largest term of the sum, then reversing the order of summation. To prove this formula directly for $\abs{\mathcal{R}(m, n)}$, we will need a lemma:

\begin{lem}
\label{lem:surj}
Fix $j < n$. The number of ways to partition $[n]$ into $n-j$ non-empty consecutive parts, then label all of the elements so that the labels within each part are in reverse order is $(n-j)!\stirling{n}{n-j}$. In other words, it is equal to the number of surjections from $[n]$ to $[n-j]$. 
\end{lem}
By consecutive, we mean the part contains all elements between its minimal and maximal part. e.g.
\[
    [123][45][6][78] \text{ is valid, but} \qquad [124][3] \text{ is not.}
\]

\begin{proof}
Consider a surjection $g\colon [n] \to [n-j]$. The size of the preimage of each $i \in [n-j]$ gives the size of the $i$'th part. e.g. if $\abs{g^{-1}(1)} = 4$, then $[1234]$ is a part. Surjectivity ensures each part is nonempty.

Then, simply label the elements from that part with the elements from the preimage in reverse order. e.g. if $g^{-1}(1) = \{4, 5, 7, 9\}$, then there would be a part $[1234]$ with labels $9754$ (in that order). 

Going from the partition/labeling to a surjection $g$ is done simply by letting the labeling of the first part be the preimage of 1, that of the second part be the preimage of $2, \ldots$. Nonempty parts ensures a surjective function. 
\end{proof}

Recall from \cref{sec:complete-bipartite} that if we let $\mathcal{R}(m, n)$ be the set of AOs of $K_{m, n}$ in which there are no sinks in $L$ (the left vertex set), then there is a simple bijection showing $\abs{\mathcal{R}(m, n)} = \abs{\mathcal{A}(K_{m + 1, n}, s)}$. 

\begin{thm}
\label{thm:bipartite-enum}
\[
    \abs{\mathcal{R}(m, n)} = n!n^m - \sum_{j = 1}^{n-1} (-1)^{j-1} \cdot (n-j)!\cdot (n-j)^{m}\cdot  \stirling{n}{n-j}.
\]
\end{thm}

\begin{proof}
Consider the graph $K'_{m, n}$, the complete bipartite graph with $m$ labeled vertices $L = \{1, \ldots, m\}$, and $n$ unlabeled vertices $R = \{r, \ldots, r\}$. 

By \cref{lem:canonical-sorts}, Counting the number of AOs on $K'_{m, n}$ such that there are no sinks in $L$ can be done by counting canonical topological sorts of the vertices such that the final vertex in the sort is from $R$. To do this, first write out all $n$ vertices from $R$. Then, for each $v \in L$, $v$ can be placed in any of the $n$ spaces between these vertices (not counting the final space, since we don't allow $v$ to be a sink). 

Once such a space is chosen, $v$'s ordering with respect to the other vertices of $L$ in the same space is uniquely determined, since we are counting canonical sorts. Thus, we have $n^m$ such AOs of $K'_{m, n}$.
\[
    \_r\_r\_r\_r \quad \longrightarrow \quad r13r2r4r
\]

Now, we want to relabel the vertices of $R$, and count the resulting number of canonical topological sorts for $K_{m, n}$. For a given canonical topological sort of $K'_{m, n}$, there are $n!$ ways to relabel $R$, but some may result in a non-canonical sort. 

Given any sort $\mathcal{O}$ for $K'_{m, n}$, we may refer to the $i$'th occurrence of $r$ as $r_i$, and for a particular relabeling $\ell$, the corresponding label is denoted $\ell(r_i)$. Also, let $r_i \sim r_{i + 1}$ denote that they are adjacent in the sort (i.e. they are incomparable. No vertex from $L$ was placed in the space between them). Then, we may count the non-canonical sort/relabeling pairs as the union of sets:
\[
    A_i = \{(\mathcal{O}, \ell) \mid r_i \sim r_{i + 1} \text{ and } \ell(r_i) > \ell(r_{i + 1})\}
\]
for each $i = 1, \ldots, n - 1$ (it is clear that every non-canonical sort/ relabeling pair must be in at least one such set). Then, inclusion-exclusion gives a way to count all non-canonical sort/relabeling pairs:
\[
    \abs{\bigcup_{i = 1}^{n-1} A_i} = \sum_{\varnothing \neq J \subseteq [n-1]} (-1)^{\abs{J} - 1} \abs{\bigcap_{i \in J} A_i} = \sum_{j = 1}^{n-1} (-1)^{j - 1} \sum_{\substack{J \subseteq [n-1] \\ \abs{J} = j}} \abs{\bigcap_{i \in J} A_i} \,.
\]
For fixed $J$, $\abs{\bigcap_{i \in J} A_i}$ is the number of sort/label pairs such that for each $i \in J$, $r_i \sim r_{i + 1}$, and within a group of incomparable vertices from $R$, the relabeling is in reverse order (by definition of $A_i$). In particular, for a fixed $j$, we would like to count all of the ways to partition $r_1, \ldots, r_n$ into $n-j$ consecutive parts (this is the choice of $J$), then label them so that within each part the labelling is in reverse order. Notice, this is exactly what \cref{lem:surj} counts. Then, after choosing a $J$ of size $j$ and the labels for the $r_i$, we must complete the orientation by choosing the placement of vertices from $L$. Since we need $r_i \sim r_{i + 1}$ for each $i \in J$, there are $j$ spaces disallowed for placement of vertices from $L$, so there are $(n-j)^m$ choices of for each $J$. In total, we get that 
\[
    \sum_{\substack{J \subseteq [n-1] \\ \abs{J} = j}} \abs{\bigcap_{i \in J} A_i} = (n-j)^m(n-j)!\stirling{n}{n-j}
\]
and therefore
\[
    \abs{\bigcup_{i = 1}^{n-1} A_i} = \sum_{j = 1}^{n-1} (-1)^{j - 1} (n-j)^m(n-j)!\stirling{n}{n-j}.
\]
Finally, as these are the non-canonical sort/labelling pairs, we subtract from the total number to get the canonical ones
\begin{align*}
\abs{\mathcal{R}(m, n)} &= n!n^m - \abs{\bigcup_{i = 1}^{n-1} A_i} \\
&=  n!n^m - \sum_{j = 1}^{n-1} (-1)^{j-1} \cdot (n-j)!\cdot (n-j)^{m}\cdot  \stirling{n}{n-j},
\end{align*}
completing the proof.
\end{proof}

\subsection{Complete Multipartite Graphs}

Here, we consider AOs of complete multipartite graphs $K_{n_1, \ldots, n_N}$. We refer to the vertex sets in the separate parts as $V_1,\dots,V_N$.
To simplify notation, we set $|k| = \sum_{i = 2}^N k_i$ for any
$k = (k_2,\dots,k_N) \in \mathbb{N}^{N-1}$.

\begin{lem}
\label{unlabeled}
Consider the complete $N$-partite graph $K'_{n_1, \ldots, n_N}$ with the vertices in each $n_2, \ldots, n_N$-set unlabeled within their vertex set. Then 
\[
    \abs{\mathcal{A}(K'_{n_1, \ldots, n_N})} = \left(1 + |n|\right)^{n_1} \cdot \binom{|n|}{n_2, \ldots, n_N}\,.
\]
\end{lem}

\begin{proof}
The proof idea is a generalization of \cref{thm:bipartite-enum}. 
We refer to the unlabeled vertex sets as $V_2, \ldots, V_N$, and the labeled $n_1$-set as $V_1$.
We will count ways to construct an AO of $K'_{n_1, \ldots, n_N}$ by counting canonical topological sorts of $K'_{n_1, \ldots, n_N}$. First the edges disjoint from $V_1$ will be oriented, then all edges containing a vertex from $V_1$ will be oriented.

To orient edges disjoint from $V_1$, we pick the relative order in the topological sort of vertices in $V_2 \cup \cdots \cup V_N$. Since these vertices are unlabeled (within their respective vertex sets), we will collectively refer to any vertex in $V_i$ with the label $v^i$. We can order them by writing a sequence containing $n_2$ copies of $v^2$, $n_3$ copies of $v^3$, and so on until $n_N$ copies of $v^N$. The number of such sequences is given by the multinomial coefficient
\[
    \binom{|n|}{n_2, \ldots, n_N}\,.
\]

Next, given any such sequence $S$, to pick the orientations for the edges containing vertices in $V_1$, we simply pick where to insert each vertex $a_1, \ldots, a_{n_1} \in V_1$ into $S$ to create a canonical topological sort. $S$ if of length $|n|$, so there are $1 + |n|$ ``spaces'' in which we can place $a_1, \ldots, a_n$ (including the spaces on the far left and far right). Notice that if two $a_i$, $a_j$ get placed in the same space, then they are incomparable---their order within that space is determined since we are counting canonical sorts. So, we simply choose a space for each vertex in $V_1$, giving $(1 + |n|)^{n_1}$ choices. 

The equivalence between canonical sorts and AOs of $K_{n_1, \ldots, n_N}$ given in \cref{lem:canonical-sorts} completes the proof.
\end{proof}

\begin{rem}
Notice that we could easily alter this proof to count AOs of $K'_{n_1, \ldots, n_N}$ such that there are no sinks in vertex set $V_1$. The only modification needed is to disallow placement of vertices from $V_1$ in the rightmost space of the topological sort. This means that there are only $|n|$ choices for each $a \in V_1$, which gives the total number of such AOs as 
\[
    |n|^{n_1} \cdot \binom{|n|}{n_2, \ldots, n_N}\,.
\]
This will allow us to easily count the number of AUSOs of multipartite graphs as well.
\end{rem}

\begin{thm}
$\abs{\mathcal{A}(K_{n_1, \ldots, n_N})}$, the number of AOs of the complete $N$-partite graph $K_{n_1, \ldots, n_N}$ is given by
\[
    \sum_{(k_2, \ldots, k_N) \in \mathcal{K}} (-1)^{|n| -|k|} 
     (1 + |k|)^{n_1} |k|!  
    \prod_{i = 2}^N \stirling{n_i}{k_i} ,
\]
where $\mathcal{K} = [n_2] \times [n_3] \times \cdots \times [n_N]$. 
\end{thm}

\begin{proof}
First, we write an equivalent formula by reversing the order of summation along each $[n_i]$. Let $\mathcal{J} = \{0,\ldots, n_2-1\} \times \cdots \times \{0,\ldots, n_N-1\}$, and make the substitutions $j_i = n_i - k_i$. We will instead show that $\abs{\mathcal{A}(K_{n_1, \ldots, n_N})}$ is given by 
\begin{multline*}
    \sum_{(j_2, \ldots, j_N) \in \mathcal{J}} (-1)^{|j|} \cdot (1 + |n| - |j|)^{n_1}  (|n| - |j|)! \cdot \prod_{i = 2}^N \stirling{n_i}{n_i - j_i}.
\end{multline*}
\cref{unlabeled} gives a way to count $\abs{\mathcal{A}(K'_{n_1, \ldots, n_N})}$, and we may consider all \textit{relabelings} of vertices in sets $V_2, \ldots, V_N$. There are $\prod_{i = 2}^N n_i!$ such relabelings, and this gives a multiset $\mathscr{A}$ of AOs of $K_{n_1, \ldots, n_N}$ of size 
\[
    \abs{\mathscr{A}} = \abs{\mathcal{A}(K'_{n_1, \ldots, n_N})} \cdot \prod_{i = 2}^N n_i!= \left(1 + |n|\right)^{n_1} \cdot \binom{|n|}{n_2, \ldots, n_N} \cdot \prod_{i = 2}^N n_i!
\]
If we denote the set of all relabelings as $\mathcal{L}$, then there is an obvious correspondence between $\mathscr{A}$ and $\mathcal{A}(K'_{n_1, \ldots, n_N}) \times \mathcal{L}$. It is also clear that eliminating duplicates in the multiset would exactly yield $\mathcal{A}(K_{n_1, \ldots, n_N})$. In particular, using the equivalence of \cref{lem:canonical-sorts}, some of the topological sorts resulting from relabeling are non-canonical. We will count how many are non-canonical in a similar manner to \cref{thm:bipartite-enum}. 

For any canonical topological sort (an orientation) $\mathcal{O}$ of $K'_{n_1, \ldots, n_N}$, we can refer to the $k$'th occurrence of $v^i$, a vertex from the unlabeled set $V_i$, as $v^i_k$. Then, if $\ell$ is a relabeling, we denote the label of this vertex as $\ell(v^i_k)$. Let $v^i_k \sim v^i_{k + 1}$ denote that these two unlabeled vertices (from the same vertex set) are adjacent in the sort. That is to say, they are incomparable. Then, a non-canonical relabeling of $\mathcal{O}$ must have some $v^i_k \sim v^i_{k + 1}$ with $\ell(v^i_k) > \ell(v^i_{k + 1})$. 

Thus, for each vertex set $V_2, \ldots, V_N$ and each vertex $v^i_k \in V_i$ consider the set of non-canonical orientations / labeling pairs:
\[
    \mathscr{B}_{V_i, k} = \{ (\mathcal{O}, \ell) \in \mathcal{A}(K'_{n_1, \ldots, n_N}) \times \mathcal{L} \mid v^i_k \sim v^i_{k + 1}, \ell(v^i_k) > \ell(v^i_{k + 1}) \}\,.
\]
The union of these sets is the set of all non-canonical orientations / labeling pairs in $\mathscr{A}$. We will count this with inclusion-exclusion. Let 
\[
    \mathcal{I} = \bigcup_{i = 2}^N \{(V_i, k) \mid 1 \leq k < n_i\}\,.
\]
Then the principle of inclusion exclusion gives that the number of non-canonical pairs is
\[
    \sum_{\varnothing \neq J \subseteq \mathcal{I}} (-1)^{\abs{J} - 1} \abs{\bigcap_{(V, k) \in J} \mathscr{B}_{V, k}}\,.
\]
Analogously to \cref{thm:bipartite-enum}, we first consider all $J$ such that $j_2$ pairs in $V_2$ are incomparable, $j_3$ pairs in $V_3$ are incomparable, $\ldots, j_N$ pairs in $V_N$ are incomparable for some fixed $j_2, \ldots, j_N$. We can look at each vertex set independently and use \cref{lem:surj} to see that in total, the number of ways to choose $j_i$ indices in vertex set $V_i$ and relabel to get get something in $\bigcap_{(V, k) \in J} \mathscr{B}_{V, k}$ is 
\[
    \prod_{i = 2}^N (n_i - j_i)! \stirling{n_i}{n_i - j_i}.
\]
Then, for each such $J$, regardless of the indices and labelling, the number of topological sorts can be counted identically to \cref{unlabeled} by first ordering the groups of incomparable vertices in $V_2 \cup \cdots \cup V_N$, then selecting a space for each $a_1, \ldots, a_{n_1} \in V_1$. There are
\[
    \left(1 + |n| - |j| \right)^{n_1} \cdot \binom{|n| - |j|}{n_2 - j_2, \ldots, n_N - j_N}  
\]
ways. Therefore, if for $J \subseteq \mathcal{I}$, we denote $J_i = J \cap \{(V_i, k) \mid 1 \leq k < n_i\}$, then for any fixed $j_2, \ldots, j_N$,
\begin{multline*}
    \sum_{\substack{J \subseteq I\\\forall i, \abs{J_i} = j_i}} \abs{\bigcap_{(V, k) \in J} \mathscr{B}_{V, k} } = (1 + |n| - |j| )^{n_1}  (|n| - |j|)! \prod_{i = 2}^N \stirling{n_i}{n_i - j_i} \,.
\end{multline*}
This lets us rewrite the inclusion-exclusion as 
\begin{multline*}
\sum_{(j_2, \ldots, j_N) \in \mathcal{J}'} (-1)^{|j|-1}
 (1 + |n| - |j|)^{n_1} (|n| - |j|)! \prod_{i = 2}^N \stirling{n_i}{n_i - j_i} ,
\end{multline*}
where $\mathcal{J}' = \mathcal{J} \setminus \{(0, \ldots, 0)\}$. Subtracting this from $\abs{\mathscr{A}}$, we get the desired result.
\end{proof}

\begin{rem}
\label{rem:stanley}
Richard Stanley kindly suggested to us that an alternative approach to this enumeration problem might be to use the generating function for the chromatic polynomial $\chi(q)$ of complete
multipartite graphs $K_{n_1,\ldots,n_k}$ for fixed $k\ge 1$~\cite[Solution to Problem~5.6]{stanley-ec2}. 
Differentiating this expression with respect to $q$ and setting $q=0$ 
to get (up to sign) the number of AOs with a fixed sink 
gives the recurrence 
\begin{multline*}
    \abs{\mathcal{A}(K_{n_1, \ldots, n_N})} + \sum_{\ell = 0}^{n_1-1} (-1)^{n_1-\ell} \binom{n_1}{\ell} \abs{\mathcal{A}(K_{\ell, n_2, \ldots, n_N})}  + \\
    \cdots + \sum_{\ell = 0}^{n_N-1} (-1)^{n_N-\ell} \binom{n_N}{\ell} \abs{\mathcal{A}(K_{n_1, \ldots, n_{N-1}, \ell})} = 0.
\end{multline*}
\end{rem}

\begin{thm}
\label{thm:auso-multi}
The number of AUSOs with fixed sink of $K_{n_1 + 1, n_2, \ldots, n_N}$ is given by
\begin{multline*}
    \abs{\mathcal{A}(K_{n_1 + 1, n_2, \ldots, n_N}, s)} =  \sum_{(k_2, \ldots, k_N) \in \mathcal{K}} (-1)^{|n| - |k|)} \cdot |k|^{n_1} |k|!
    \prod_{i = 2}^N \stirling{n_i}{k_i} \,.
\end{multline*}
\end{thm}
\begin{proof}
The proof is identical by just counting the number of AOs of $K_{n_1, \ldots, n_N}$ which have no sinks in vertex set $V_1$, and using the remark after \cref{unlabeled}.
\end{proof}

\section{Tur\'an graphs and Extremality}
\label{sec:cameron}

Recall that the {\em Tur\'an graph} $T(n,r)$ is a graph on $n$ vertices, where the vertex set is partitioned into $r$ subsets of sizes as close to each other as possible, and edges connecting vertices if and only if they belong to different subsets. To be precise, it is the complete multipartite graph of the form $K_{\lceil n/r \rceil, \lceil n/r \rceil, \dots, \lfloor n/r \rfloor,  \lfloor n/r \rfloor}$ with $r$ parts.

Let $u_{n,r} = \abs{\mathcal{A}(T(n, r), s)}$ denote the number of acyclic orientations with unique sink for the Tur\'an graph $T(n,r)$. 
Given a sequence $\{a_n\}$, define the difference operator $\delta$ by $\delta(\{a_n\}) = \{a_n - a_{n-1}\}$.

\begin{thm}
\label{thm:aou-turan}
Fix an integer $k \geq 0$. Then $u_{r+k,r}$ is given by the $(r+k)$'th entry of the sequence $\delta^k \{(r+k-1)!\}$ provided $r \geq k$. In particular, for such $r$, $u_{r+k,r} = u_{r+k,r+1} - u_{r+k-1,r}$.
\end{thm}

Some data for the number of acyclic orientations with fixed sink is given in \cref{tab:turan}. Clearly, the rightmost column is the sequence $\{(n-1)!\}$ because $T(n,n)$ is the complete graph $K_n$ and there are $(n-1)!$ such acyclic orientations. 
The column corresponding to $r=2$ is covered by the discussion on bipartite graphs in \cref{sec:complete-bipartite}.

\begin{table}[htbp!]
    \centering
    \begin{tabular}{c|ccccccc}
        $n$\textbackslash $r$ & 1 & 2 & 3 & 4 & 5 & 6 & 7 \\
        \hline
        1 & \underline{1} \\
        2 & \underline{0} & \underline{1} \\
        3 & 0 & \underline{1} & \underline{2} \\
        4 & 0 & \underline{3} & \underline{4} & \underline{6} \\
        5 & 0 & 7 & \underline{14} & \underline{18} & \underline{24} \\
        6 & 0 & 31 & \underline{64} & \underline{78} & \underline{96} & \underline{120} \\
        7 & 0 & 115 & 284 & \underline{426} & \underline{504} & \underline{600} & \underline{720}
    \end{tabular}
    \caption{The number of  acyclic orientations with fixed sink of $T(n,r)$ for small values of $n$ and $r$. The underlined numbers fall under the purview of \cref{thm:aou-turan}.}
    \label{tab:turan}
\end{table}{}

To prove \cref{thm:aou-turan}, we will appeal to a basic result about AUSOs. For a graph $G$ containing an edge $e$, let $G \setminus e$ denote the graph $G$ with the edge $e$ removed and $G / e$ denote the graph $G$ with the edge $e$ contracted. For a bijective proof of the following lemma, see Lemma~2.2 in \cite{gebhard2000sinks}.

\begin{lem}
\label{lem:aou-del-con}
For any graph $G$, we have that
\[
    \abs{\mathcal{A}(G, s)} = \abs{\mathcal{A}(G \setminus e, s)} + \abs{\mathcal{A}(G / e, s)}.
\]
\end{lem}

\begin{proof}[Proof of \cref{thm:aou-turan}]
For $r \geq k$, the graph $T_{r+k,r}$ has $r$ parts with the first $k$ parts having two vertices each and the remainder having one. Now, consider the graph $G$ obtained by adding the edge $e$ joining the two vertices in the $k$'th part. The key observations are that $G = T_{r+k,r+1}$ and that $G / e = T_{r+k-1,r}$. By \cref{lem:aou-del-con}, we find that $u_{r+k,r} = u_{r+k,r+1} - u_{r+k-1,r}$. By appealing to the case of $k=0$, for which the result is $(r-1)!$, we obtain the proof.

Another way to prove \cref{thm:aou-turan} is starting with \cref{thm:auso-multi}
for the case $n_1 = \cdots = n_k = 2$ and $n_{k+1} = \cdots = n_r = 1$. Then the sum simplifies to
\begin{equation}
\label{uformula}
u_{r+k,r} =  \sum_{j_2 = 1}^2 \cdots \sum_{j_k = 1}^2 (-1)^{j_2 + \cdots + j_k}
(j_2 + \cdots + j_k + r-k) (j_2 + \cdots + j_k + r-k)!.
\end{equation}
Now, using \eqref{uformula}, we can compute $u_{r+k,r} + u_{r+k-1,r}$ by summing over the $j_k$ variable to obtain
\begin{multline*}
-(j_2 + \cdots + j_{k-1} + r-k+1) (j_2 + \cdots + j_{k-1} + r-k+1)! \\
+(j_2 + \cdots + j_{k-1} + r-k+2) (j_2 + \cdots + j_{k-1} + r-k+2)!
\end{multline*}
The first term cancels exactly with the formula for $u_{r+k-1,r}$ and the second one is 
$u_{r+k,r+1}$, proving the result.
\end{proof}

\subsection{Extremal Graphs for AOs}
\label{sec:extremal}

We conclude this contribution by addressing the following extremal problem on the number of acyclic orientations: given a fixed number $n$ of vertices, and $m$ of edges, which graph(s) maximize the number of AOs? 

It is conjectured by Cameron, Glass, and Schumacher~\cite{cameron2014acyclic} that Tur\'an graphs with two parts maximize the number of AOs over graphs with the same number of vertices and edges. We prove here that Tur\'an graphs with parts of size at most 2 are also maximizers. 

\begin{lem}
\label{slide}
Let $G$ be any graph containing an edge $e = (a, b)$ such that 
\[
    N(a) \setminus \{b\} \supseteq N(b) \setminus \{a\},
\]
where $N$ denotes the set of neighboring vertices. Then, for any vertex $c$ such that the edge $e' = (c, b)$ is not in $G$, $\abs{\mathcal{A}(G)} \leq \abs{\mathcal{A}(G \setminus e + e')}$.

\begin{proof}
Let $G' = G \setminus e + e'$. We know by the deletion-contraction recurrence for AOs that $\abs{\mathcal{A}(G)} = \abs{\mathcal{A}(G \setminus e)} + \abs{\mathcal{A}(G / e)}$ (equivalently for $G'$ and $e'$). But notice that $G \setminus e = G' \setminus e'$, so we just need show that $\abs{\mathcal{A}(G / e)} \leq \abs{\mathcal{A}(G' / e')}$

First, $G / e \cong G \setminus \{b\}$, since the new vertex $ab$ resulting from contracting $e = (a, b)$ has neighborhood 
\[
    N(ab) = \left( N(a) \setminus \{b\} \right) \cup \left( N(b) \setminus \{a\} \right) = N(a) \setminus \{b\}\,,
\]
while the rest of $G$ is unchanged. On the other hand, $G' / e'$ has new vertex $bc$ resulting from contracting $e' = (c, b)$ with the neighborhood 
\[
    N(bc) = \left( N(b) \setminus \{c\} \right) \cup \left( N(c) \setminus \{b\} \right) \supseteq N(c) \setminus \{b\}\,.
\]
Therefore, $G' / e' \supseteq G \setminus \{b\}$, as the neighborhood of $bc$ is no smaller than that of $c$. This gives the desired
\[
    G / e \subseteq G' / e' \implies \abs{\mathcal{A}(G/ e)} \leq \abs{\mathcal{A}(G'/ e')} \implies \abs{\mathcal{A}(G)} \leq \abs{\mathcal{A}(G')}\,. \qedhere
\]
\end{proof}
\end{lem}

\begin{thm}
\label{thm:extremal}
For $m \geq \binom{n}{2} - \floor{\frac{n}{2}}$, the graph whose complement is a matching maximizes the number of AOs over all graphs on the same number of vertices and edges.

\begin{proof}
\cref{slide} can be interpreted in the complement: if there is an edge $e = (c, b)$ in $\overline{G}$, and a vertex $a$ with $N(a) \subseteq N(b)$, then we can replace $e$ with $e' = (a, b)$ in the complement without decreasing the number of AOs in $G$. 

In particular, if $a$ is an isolated vertex, we can always ``slide" any edge to $a$. If there are at most $\floor{\frac{n}{2}}$ edges in the complement and it is not a matching, then it has an isolated vertex. So for any graph whose complement is not a matching, there is a series of edge slides which do not decrease the number of AOs, eventually resulting in a graph whose complement is a matching.
\end{proof}
\end{thm}

\begin{rem}
The complement of a matching is a Tur\'an graph with parts of sizes 1 or 2, so \cref{thm:extremal} proves the claim.
\end{rem}

\section*{Acknowledgements}
We would like to thank B. B\'enyi, D. Grinberg and R. Stanley for insightful comments.
AA would like to thank M. Ravichandran for discussions on toppleable permutations while he was in residence at Institut Mittag-Leffler in Djursholm, Sweden during the semester of Spring, 2020.
AA was partially supported by a UGC Centre for Advanced Study grant, by Department of Science and Technology grant EMR/2016/ 006624, and by grant no. 2016-06596 of the Swedish Research Council. 
DH was partially supported by Georgia Tech Mathematics REU NSF grant DMS-1851843. 
The research of PT is supported in part by the NSF grant DMS-1811935.

\bibliographystyle{alpha}
\bibliography{bibliography}

\end{document}